\documentclass[10pt,a4paper,twoside]{article}
\usepackage{amsthm,amsfonts,amssymb,amsmath,amscd,bezier}      
\usepackage{graphicx}           

\usepackage[all]{xy}

\newcommand{\R}{\ensuremath{\mathbb{R}}}

\newcommand{\Z}{\ensuremath{\mathbb{Z}}}
\newcommand{\N}{\ensuremath{\mathbb{N}}}

\newcommand{\RP}{\ensuremath{\mathbb{R}\mathrm{P}}}

\newcommand{\SIF}{\ensuremath{\{X_{i},Y_{i},f_{i}\}}}
\newcommand{\SIFI}{\ensuremath{\{X_{i},X_{i},f_{i}\}}}
\newcommand{\funtor}{\ensuremath{\mathbf{F}}}
\newcommand{\contrafuntor}{\ensuremath{\mathbf{G}}}
\newcommand{\X}{\ensuremath{\mathfrak{X}}}
\newcommand{\Ze}{\ensuremath{\mathfrak{Z}}}

\newtheorem{Defn}{Definition}[section]
\newtheorem{prop}[Defn]{Proposition}
\newtheorem{lem}[Defn]{Lemma}
\newtheorem{Rem}[Defn]{Remark}
\newtheorem{thm}[Defn]{Theorem}
\newtheorem{cor}[Defn]{Corollary}
\newtheorem{Exmp}[Defn]{Example}

\begin{document}

\thispagestyle{empty}

\begin{center}

{\LARGE Closed injective systems and its fundamental limit spaces}

\vspace{7mm}

{\large Marcio Colombo Fenille\footnote{E-mail:fenille@icmc.usp.br}}

\vspace{7mm}

Instituto de Ci\^ encias Matem\'aticas e de Computa\c c\~ao -
Universidade de S\~ao Paulo - Av. Trabalhador S\~ao-Carlense, 400,
Centro, Caixa Postal 668, CEP 13560-970, S\~ao Carlos, SP, Brazil.

\vspace{3mm}



\end{center}


\noindent {\bf Abstract:} In this article we introduce the concept
of limit space and fundamental limit space for the so-called closed
injected systems of topological spaces. We present the main results
on existence and uniqueness of limit spaces and several concrete
examples. In the main section of the text, we show that the closed
injective system can be considered as objects of a category whose
morphisms are the so-called cis-morphisms. Moreover, the transition
to fundamental limit space can be considered a functor from this
category into category of topological spaces. Later, we show results
about properties on functors and counter-functors for inductive
closed injective system and fundamental limit spaces. We finish with
the presentation of some results of characterization of fundamental
limite space for some special systems and the study of so-called
perfect properties.

\vspace{2mm}



\begin{center}




\end{center}



\vspace{-10mm}

\section{Introduction}


Our purpose is to introduce and study what we call category of
closed injective systems and cis-morphisms, beyond the limit spaces
of such systems.


We start by defining the so-called closed injective systems (CIS to
shorten), and the concepts of limit space for such systems. We have
particular interest in a special type of limit space, those we call
fundamental limit space. Section \ref{Section.Fundamental.L.E.} is
devoted to introduce this concept and demonstrate theorems of
existence and uniqueness of fundamental limit spaces. The following
section, in turn, is devoted to presenting some very illustrative
examples.


Section \ref{Section.Category} is one of the most important and
interesting for us. There we show that a closed injective system can
be considered as object of a category, whose morphisms are the
so-called cis-morphisms, which we define in this occasion.
Furthermore, we prove that this category is complete with respect to
direct limits, that is, all inductive system of CIS's and
cis-morphisms has a direct limit.


In Section \ref{Section.Passagem.ao.limite}, we prove that the
transition to the fundamental limit can be considered as a functor
from category of CIS's and cis-morphisms into category of
topological spaces and continuous maps.


In Section \ref{Section.Compatibility.Limits}, we show that the
transition to the direct limit in the category of CIS's and
cis-morphisms is compatible (in a way) to transition to the
fundamental limit space.


In section \ref{Section.Inductive.System}, we study a class of
special CIS's called inductive closed injective systems. In the two
following sections, we study the action of functors and
counter-functors, respectively, in such systems, and present some
simple applications of the results demonstrated.

We finish with the presentation of some results of characterization
of fundamental limite space for some special systems, the so-called
finitely-semicomponible and stationary systems, and the study of
so-called perfect properties over topological spaces of a system and
over its fundamental limit spaces.

\section{Closed injective system and limit spaces}\label{Section.CIS} 

Let $\{X_{i}\}_{i=0}^{\infty}$ be a countable collection of nonempty
topological spaces. For each $i\in\N$, let $Y_{i}$ be a closed
subspace of $X_{i}$. Assume, for each $i\in\N$, there is a closed
injective continuous map
$$f_{i}:Y_{i}\rightarrow X_{i+1}.$$ This structure is called {\bf
closed injective system}, or CIS, to shorten. We write
$\{X_{i},Y_{i},f_{i}\}$ to represent this system. Moreover, by
injection we mean a injective continuous map.

We say that two injection $f_{i}$ and $f_{i+1}$ are {\bf
semicomponible} if $f_{i}(Y_{i})\cap Y_{i+1}\neq\emptyset$. In this
case, we can define a new injection
$$f_{i,i+1}:f_{i}^{-1}(Y_{i+1})\rightarrow X_{i+2}$$ by $f_{i,i+1}(y)=(f_{i+1}\circ f_{i})(y)$, for all $y\in
f_{i}^{-1}(Y_{i+1})$.

For convenience, we put $f_{i,i}=f_{i}$. Moreover,  we say that
$f_{i}$ is always semicomponible with itself. Also, we write
$f_{i,i-1}$ to be the natural inclusion of $Y_{i}$ into $X_{i}$, for
all $i\in\N$.

Given $i,j\in\N$, $j>i+1$, we say that $f_{i}$ and $f_{j}$ are {\bf
semicomponible} if $f_{i,k}$ and $f_{k+1}$ are semicomponible for
all $i+1\leq k\leq j-1$, where
$$f_{i,k}:f_{i,k-1}^{-1}(Y_{k})\rightarrow X_{k+1}$$ is defined
inductively. To facilitate the notations, se $f_{i}$ and $f_{j}$ are
semicomponible, we write $$Y_{i,j}=f_{i,j-1}^{-1}(Y_{j}),$$ that is,
$Y_{i,j}$ is the domain of the injection $f_{i,j}$. According to the
agreement $f_{i,i}=f_{i}$, we have $Y_{i,i}=Y_{i}$.

\begin{lem}\label{Lema.Inicial.1}
If $f_{i}$ and $f_{j}$ are semicomponible, $i<j$, then $f_{k}$ and
$f_{l}$ are semicomponible, for any integers $k,l$ with $i\leq k\leq
l\leq j$.
\end{lem}

\begin{lem}
If $f_{i}$ and $f_{j}$ are not semicomponible, then $f_{i}$ and
$f_{k}$ are not semicomponible, for any integers $k>j$.
\end{lem}

\begin{lem}\label{Lema.Inicial.2}
Assume that $f_{i}$ and $f_{j}$ are semicomponible, with $i<j$. Then
we have
$$Y_{i,j}=(f_{j-1}\circ\cdots\circ f_{i})^{-1}(Y_{j}) \ \ \ {\rm and}
\ \ \ f_{i,j}(Y_{i,j})=(f_{j}\circ f_{i,j-1})(Y_{i,j-1}).$$
\end{lem}

The proofs of above results are omitted.


Henceforth, since products of maps do not appear in this paper, we
can sometimes omit the symbol $\circ$ in the composition of maps.

\begin{Defn}
Let $\SIF$ be a CIS. A {\bf limit space} for this system is a
topological space $X$ and a collection of continuous maps
$\phi_{i}:X_{i}\rightarrow X$ satisfying the following conditions:
\begin{enumerate}
\item[{L.1.}] $X=\bigcup_{i=0}^{\infty}\phi_{i}(X_{i})$;
\item[{L.2.}] Each $\phi_{i}:X_{i}\rightarrow X$ is a imbedding;
\item[{L.3.}] $\phi_{i}(X_{i})\cap\phi_{j}(X_{j})\doteq\phi_{j}f_{i,j-1}(Y_{i,j-1})$ if $i<j$ and $f_{i}$ and $f_{j}$ are
semicomponible;
\item[{L.4.}] $\phi_{i}(X_{i})\cap\phi_{j}(X_{j})=\emptyset$ if $f_{i}$ and $f_{j}$ are not semicomponible;
\end{enumerate}
where $\doteq$ indicates, besides the equality of sets, the
following: If $x\in\phi_{i}(X_{i})\cap\phi_{j}(X_{j})$, say
$x=\phi_{i}(x_{i})=\phi_{j}(x_{j})$, with $x_{i}\in X_{i}$ and
$x_{j}\in X_{j}$, then we have necessarily $x_{i}\in Y_{i,j-1}$ and
$x_{j}=f_{i,j-1}(x_{i})$.
\end{Defn}

\begin{Rem}
The ``pointwise identity'' indicated by $\doteq$ {\rm L.3} reduced
to identity of sets indicates only that
$$\phi_{i}(X_{i})\cap\phi_{j}(X_{j})=\phi_{i}(Y_{i,j-1})\cap\phi_{j}f_{i,j-1}(Y_{i,j-1}).$$
\end{Rem}

The existence of different interpretations of the condition L.3 is
very important. Furthermore, equivalent conditions to those of the
definition can be very useful. The next results give us some
practical interpretations and equivalences.

\begin{lem}\label{Lema.Inicial.3}
Let $\{X,\phi_{i}\}$ be a limit space for the CIS $\SIF$ and suppose
that $f_{i}$ and $f_{j}$ are semicomponible, with $i<j$. Then
$\phi_{j}f_{i,j-1}(y_{i})=\phi_{i}(y_{i})$, for all $y_{i}\in
Y_{i,j-1}$.
\end{lem}
\begin{proof}
Let $y_{i}\in Y_{i,j-1}$ be an arbitrary point. By condition L.3 we
have $\phi_{j}f_{i,j-1}(y_{i})\in\phi_{i}(X_{i})$, that is,
$\phi_{j}f_{i,j-1}(y_{i})=\phi_{i}(x_{i})$ for some $x_{i}\in
X_{i}$. Again, by condition L.3, $x_{i}\in Y_{i,j-1}$ and
$f_{i,j-1}(x_{i})=f_{i,j-1}(y_{i})$. Since each $f_{k}$ is
injective, $f_{i,j-1}$ is injective, too. Therefore $x_{i}=y_{i}$,
which implies $\phi_{j}f_{i,j-1}(y_{i})=\phi_{i}(y_{i})$.
\end{proof}

\begin{lem}\label{Lema.Inicial.4}
Let $\{X,\phi_{i}\}$ be a limit space for the CIS $\SIF$, and
suppose that $f_{i}$ and $f_{j}$ are semicomponible, with $i<j$.
Then
$$\phi_{i}(X_{i}-Y_{i,j-1})\cap\phi_{j}(X_{j}-f_{i,j-1}(Y_{i,j-1}))=\emptyset.$$
\end{lem}
\begin{proof}
It is obvious that if
$x\in\phi_{i}(X_{i}-Y_{i,j-1})\cap\phi_{j}(X_{j}-f_{i,j-1}(Y_{i,j-1}))$
then
$x\in\phi_{i}(X_{i})\cap\phi_{j}(X_{j})\doteq\phi_{j}f_{i,j-1}(Y_{i,j-1})$.
But this is a contradiction, since $\phi_{j}$ is an imbedding, and
so
$\phi_{j}(X_{j}-f_{i,j-1}(Y_{i,j-1}))=\phi_{j}(X_{j})-\phi_{j}f_{i,j-1}(Y_{i,j-1})$.
\end{proof}

\begin{prop}\label{Proposicao.Condicoes.Para.E.L.F.}
Let $\SIF$ be an arbitrary CIS and let $\phi_{i}:X_{i}\rightarrow X$
be imbedding into a topological space
$X=\cup_{i=0}^{\infty}\phi_{i}(X_{i})$, such that:
\begin{enumerate}
\item[{\rm L.4.}] $\phi_{i}(X_{i})\cap\phi_{j}(X_{j})=\emptyset$ always
that $f_{i}$ and $f_{j}$ are not semicomponible;
\item[{\rm L.5.}] $\phi_{j}f_{i,j-1}(y_{i})=\phi_{i}(y_{i})$ for all $y_{i}\in Y _{i,j-1}$, always that $f_{i}$ and $f_{j}$ are
semicomponible, with $i<j$;
\item[{\rm L.6.}] $\phi_{i}(X_{i}-Y_{i,j-1})\cap\phi_{j}(X_{j}-f_{i,j-1}(Y_{i,j-1}))=\emptyset$,
always that $f_{i}$ and $f_{j}$ are semicomponible, $i<j$.
\end{enumerate}
Then $\{X,\phi_{i}\}$ is a limit space for the CIS $\SIF$.
\end{prop}
\begin{proof}
We prove that the condition L.3 is true. Suppose that $f_{i}$ and
$f_{j}$ are semicomponible, with $i<j$. By the condition L.5, the
sets $\phi_{i}(X_{i})\cap\phi_{j}(X_{j})$ and
$\phi_{j}f_{i,j-1}(Y_{i,j-1})$ are nonempty. We will prove that they
are pointwise equal.

Let $x\in\phi_{i}(X_{i})\cap\phi_{j}(X_{j})$, say
$x=\phi_{i}(x_{i})=\phi_{j}(x_{j})$ with $x_{i}\in X_{i}$ and
$x_{j}\in X_{j}$. Suppose, by contradiction, that $x_{i}\notin
Y_{i,j-1}$. Then $\phi_{i}(x_{i})\in\phi_{i}(X_{i}-Y_{i,j-1})$. By
the condition L.6 we must have
$\phi_{j}(x_{j})=\phi_{i}(x_{i})\notin\phi_{j}(X_{j}-f_{i,j-1}(Y_{i,j-1}))$,
that is, $\phi_{j}(x_{j})\in\phi_{j}f_{i,j-1}(Y_{i,j-1})$. So
$x_{j}\in f_{i,j-1}(Y_{i,j-1})$. Thus, there is $y_{i}\in Y_{i,j-1}$
such that $f_{i,j-1}(y_{i})=x_{j}$. By the condition L.5,
$\phi_{i}(y_{i})=\phi_{j}f_{i,j-1}(y_{i})=\phi_{j}(x_{j})$. However,
$\phi_{j}(x_{j})=\phi_{i}(x_{i})$. It follows that
$\phi_{i}(y_{i})=\phi_{i}(x_{i})$, and so $x_{i}=y_{i}\in
Y_{i,j-1}$, which is a contradiction. Therefore $x_{i}\in
Y_{i,j-1}$.

In order to prove the remaining, take
$x\in\phi_{i}(X_{i})\cap\phi_{j}(X_{j})$,
$x=\phi_{i}(y_{i})=\phi_{j}(x_{j})$, with $y_{i}\in Y_{i,j-1}$ and
$x_{j}\in X_{j}$. We must prove that $x_{j}=f_{i,j-1}(y_{i})$. By
the condition L.5,
$\phi_{j}f_{i,j-1}(y_{i})=\phi_{i}(y_{i})=\phi_{j}(x_{j})$. Thus,
the desired identity is obtained by injectivity.

This proves that
$\phi_{i}(X_{i})\cap\phi_{j}(X_{j})\doteq\phi_{j}f_{i,j-1}(Y_{i,j-1})$
and, so, that $\{X,\phi_{i}\}$ is a limite space for $\SIF$.
\end{proof}

\newpage

\begin{cor}
The condition {\rm L.3} can be replaced by both together conditions
{\rm L.5} and {\rm L.6}.
\end{cor}
\begin{proof}
The Lemmas \ref{Lema.Inicial.3} e \ref{Lema.Inicial.4} and
Proposition \ref{Proposicao.Condicoes.Para.E.L.F.} implies that.
\end{proof}

\begin{thm} \label{Teo-Bijecao}
Let $\SIF$ be a CIS. Assume that $\{X,\phi_{i}\}$ and
$\{Z,\psi_{i}\}$ are two limit spaces for this CIS. Then there is a
unique bijection (not necessarily continuous) $\beta:X\rightarrow Z$
such that $\psi_{i}=\beta\circ\phi_{i}$, for all $i\in\N$.
\end{thm}
\begin{proof}
Define $\beta:X\rightarrow Z$ in the follow way: For each $x\in X$,
we have $x=\phi_{i}(x_{i})$, for some $x_{i}\in X_{i}$. Then, we
define $\beta(x)=\psi_{i}(x_{i})$. We have:

$\diamond$ {\it $\beta$ is well defined.} Let $x\in X$ be a point
with $x=\phi_{i}(x_{i})=\phi_{j}(x_{j})$, where $x_{i}\in X_{i}$,
$x_{j}\in X_{j}$ and $i<j$. Then
$x\in\phi_{i}(X_{i})\cap\phi_{j}(X_{j})\doteq\phi_{j}f_{i,j-1}(Y_{i,j-1})$
and $x_{j}=f_{i,j-1}(x_{i})$ by the condition L.3. Thus
$\psi_{j}(x_{j})=\psi_{j}f_{i,j-1}(x_{i})=\psi_{i}(x_{i})$, where
the latter identity follows from the condition L.3.

$\diamond$ {\it $\beta$ is injective.} Suppose that
$\beta(x)=\beta(y)$, $x,y\in X$. Consider $x=\phi_{i}(x_{i})$ and
$y=\phi_{j}(y_{j})$, $x_{i}\in X_{i}$, $y_{j}\in X_{j}$, $i<j$ (the
case where $j<i$ is symmetrical and the case where $i=j$ is
trivial). Then $\psi_{i}(x_{i})=\beta(x)=\beta(y)=\psi_{j}(y_{j})$.
It follows that
$\psi_{i}(x_{i})=\psi_{j}(y_{j})\in\psi_{i}(X_{i})\cap\psi_{j}(X_{j})\doteq\psi_{j}f_{i,j-1}(Y_{i,j-1})$.
By the condition L.3, $x_{i}\in Y_{i,j-1}$ and
$y_{j}=f_{i,j-1}(x_{i})$. By the condition L.5, it follows that
$\phi_{i}(x_{i})=\phi_{j}f_{i,j-1}(x_{i})=\phi_{j}(y_{j})$.
Therefore $x=y$.

$\diamond$ {\it $\beta$ is surjective.} Let $z\in Z$ be an arbitrary
point. Then $z=\psi_{i}(x_{i})$ for some $x_{i}\in X_{i}$. Take
$x=\phi_{i}(x_{i})$, and we have $\beta(x)=z$.

The uniqueness is trivial.
\end{proof}

\section{The fundamental limit space}\label{Section.Fundamental.L.E.} 

\begin{Defn}
Let $\{X,\phi_{i}\}$ be a limit space for the CIS $\SIF$. We say $X$
has the {\bf weak topology} (induced by collection
$\{\phi_{i}\}_{i\in\N}$) if the following sentence is true:
\begin{center}
$A\subset X$ is closed in $X$ $\Leftrightarrow$ $\phi_{i}^{-1}(A)$
is closed in $X_{i}$ for all $i\in\N$.
\end{center}

\noindent When this occurs, we say that $\{X,\phi_{i}\}$ is a {\bf
fundamental limit space} for the CIS $\SIF$.
\end{Defn}

\begin{prop} \label{Prop.Implica.Fechado}
Let $\{X,\phi_{i}\}$ be a fundamental limit space for the CIS
$\SIF$. Then $\phi_{i}(X_{i})$ is closed in $X$, for all $i\in\N$.
\end{prop}
\begin{proof} We prove that $\phi_{j}^{-1}(\phi_{i}(X_{i}))$ is closed in
$X_{j}$ for any $i,j\in\N$. We have $$\phi_{j}^{-1}(\phi_{i}(X_{i}))=\left\{\begin{array}{ccl} X_{i} & \mbox{if} & i=j \\
\emptyset & \mbox{if} & i<j \ {\rm and} \  f_{i} \ {\rm and} \ f_{j} \ \mbox{ are not semicomponible}\\
\emptyset & \mbox{if} & i>j \ {\rm and} \ f_{j} \ {\rm and} \ f_{i} \ \mbox{ are not semicomponible}\\
f_{i,j-1}(Y_{i,j-1}) & \mbox{if} & i<j \ {\rm and} \ f_{i} \ {\rm and} \ f_{j} \ \, \mbox{are semicomponible} \\
f_{j,i-1}(Y_{j,i-1}) & \mbox{if} & i>j \ {\rm and} \ f_{j} \ {\rm and} \ f_{i} \ \, \mbox{are semicomponible} \\
\end{array} \right..$$

In the first three cases is obvious that
$\phi_{j}^{-1}(\phi_{i}(X_{i}))$ is closed in $X_{j}$. In the fourth
case we have the following: If $j=i+1$, then
$f_{i,j-1}(Y_{i,j-1})=f_{i}(Y_{i})$, which is closed in $X_{i+1}$,
since $f_{i}$ is a closed map. For $j>i+1$, since $f_{i}$ is
continuous and $Y_{i+1}$ is closed in $X_{i+1}$, them
$Y_{i,i+1}=f_{i}^{-1}(Y_{i+1})$ is closed in $X_{i}$. Thus, since
$f_{i}$ is closed, the Lemma \ref{Lema.Inicial.2} shows that
$f_{i,i+1}(Y_{i,i+1})=f_{i+1}f_{i}(Y_{i,i})=f_{i+1}f_{i}(Y_{i})$,
which is closed in $X_{i+1}$. Again by the Lemma
\ref{Lema.Inicial.2} we have
$f_{i,j-1}(Y_{i,j-1})=f_{j-1}f_{i,j-2}(Y_{i,j-2})$. Thus, by
induction it follows that $f_{i,j-1}(Y_{i,j-1})$ is closed in
$X_{j}$. The fifth case is similar to the fourth.
\end{proof}

\begin{cor}
Let $\{X,\phi_{i}\}$ be a fundamental limit space for the CIS
$\SIF$. If $X$ is compact, then each $X_{i}$ is compact.
\end{cor}
\begin{proof}
Each $X_{i}$ is homeomorphic to closed subspace $\phi_{i}(X_{i})$ of
$X$.
\end{proof}

\begin{prop} \label{Prop.Continuidade}
Let $\{X,\phi_{i}\}$ and $\{Z,\psi_{i}\}$ be two limit spaces for
the CIS $\SIF$. If $\{X,\phi_{i}\}$ is a fundamental limit space,
then the bijection $\beta:X\rightarrow Z$ of the Theorem
\ref{Teo-Bijecao} is continuous.
\end{prop}
\begin{proof}
Let $A$ be a closed subset of $Z$. We have
$\beta^{-1}(A)=\cup_{i=0}^{\infty}\phi_{i}(\psi_{i}^{-1}(A))$ and
$\phi_{j}^{-1}(\beta^{-1}(A))=\psi_{j}^{-1}(A)$. Since $\psi_{j}$ is
continuous and $X$ has the weak topology, we have that
$\beta^{-1}(A)$ is closed in $X$.
\end{proof}

\begin{thm} \label{Unicidade} {\sc (uniqueness of the fundamental limit space)}
Let $\{X,\phi_{i}\}$ and $\{Z,\psi_{i}\}$ be two fundamental limit
spaces for the CIS $\SIF$. Then, the bijection $\beta:X\rightarrow
Z$ of the Theorem \ref{Teo-Bijecao} is a homeomorphism. Moreover,
$\beta$ is the unique homeomorphism from $X$ onto $Z$ such that
$\psi_{i}=\beta\circ\phi_{i}$, for all $i\in\N$.
\end{thm}
\begin{proof}
Let $\beta':Z\rightarrow X$ be the inverse map of the bijection
$\beta$. By preceding proposition, $\beta$ and $\beta'$ are both
continuous maps. Therefore $\beta$ is a homeomorphism. The
uniqueness is the same of the Theorem \ref{Teo-Bijecao}.
\end{proof}

\begin{thm}\label{Teo.Existencia} {\sc (existence of fundamental limit space)}
Every closed injective system has a fundamental limit space.
\end{thm}
\begin{proof}
Let $\SIF$ be an arbitrary CIS. Define
$\widetilde{X}=X_{0}\cup_{f_{0}}X_{1}\cup_{f_{1}}X_{2}\cup_{f_{2}}\cdots$
to be the quotient space obtained of the coproduct (or topological
sum) $\coprod_{i=0}^{\infty}X_{i}$ by identifying each $Y_{i}\subset
X_{i}$ with $f_{i}(Y_{i})\subset X_{i+1}$. Define each
$\widetilde{\varphi}_{i}:X_{i}\rightarrow\widetilde{X}$ to be the
projection from $X_{i}$ into quotient space $\widetilde{X}$. Then
$\{\widetilde{X},\widetilde{\varphi}_{i}\}$ is a fundamental limit
space for the given CIS $\SIF$.
\end{proof}

The latter two theorems implies that every CIS has, up to
homeomorphisms, a unique fundamental limit space. This will be
remembered and used many times in the article.

\section{Examples of CIS's and limit spaces}\label{Section.Examples} 

In this section we will show some interesting examples of limit
spaces. The first example is very simple and the second shows the
existence of a limit space which is not a fundamental limit space.
This example will be highlighted in the last section of this article
by proving the essentiality of certain assumptions in the
characterization of the fundamental limit space through the
Hausdorff axiom. The other examples show known spaces as fundamental
limit spaces.

\begin{Exmp}
Identity limit space.

Let $\SIF$ be the CIS with $Y_{i}=X_{i}=X$ and $f_{i}=id_{X}$, for
all $i\in\N$, where $X$ is an arbitrary topological space and
$id_{X}:X\rightarrow X$ is the identity map. It is easy to see that
$\{X,id_{X}\}$ is a fundamental limite space for $\SIF$.
\end{Exmp}

\begin{Exmp} \label{Exemplo.nao.fundamental}
Existence of limit space which is not a fundamental limit space.

Assume $X_{0}=[0,1)$ and $Y_{0}=\{0\}$. Take $X_{i}=Y_{i}=[0,1]$,
for all $i\geq1$. Let $f_{0}:Y_{0}\rightarrow X_{1}$ be the
inclusion $f(0)=0$ and $f_{i}=identity$, for all $i\geq1$.

Consider the sphere $S^1$ as a subspace of $\R^2$. Define
$$\phi_{0}:X_{0}\rightarrow S^1, \ {\rm by} \ \phi_{0}(t)=(\cos\pi
t,-\sin\pi t) \ {\rm and}$$
$$\phi_{i}:X_{i}\rightarrow S^1, \ {\rm by} \ \phi_{i}(t)=(\cos\pi
t,\sin\pi t), \ {\rm for \ all} \ i\geq1.$$

\newpage

It is easy to see that $S^1=\bigcup_{i=0}^{\infty}\phi_{i}(X_{i})$
and each $\phi_{i}$ is an imbedding onto its image. Moreover,
$\phi_{i}(X_{i})\cap\phi_{j}(X_{j})\doteq\phi_{j}f_{i,j-1}(Y_{i})$,
which implies the condition {\rm L.3}.

Therefore, $\{S^1,\phi_{i}\}$ is a limit space for the CIS $\SIF$.
However, this limit space is not a fundamental limit space, since
$\phi_{0}(X_{0})$ is not closed in $S^1$, (or again, since $S^1$ is
compact though $X_{0}$ is not). (See Figure 1 below).

\begin{figure}[!htp]
\begin{minipage}[b]{0.46\linewidth}
\centering
\includegraphics[scale=0.23]{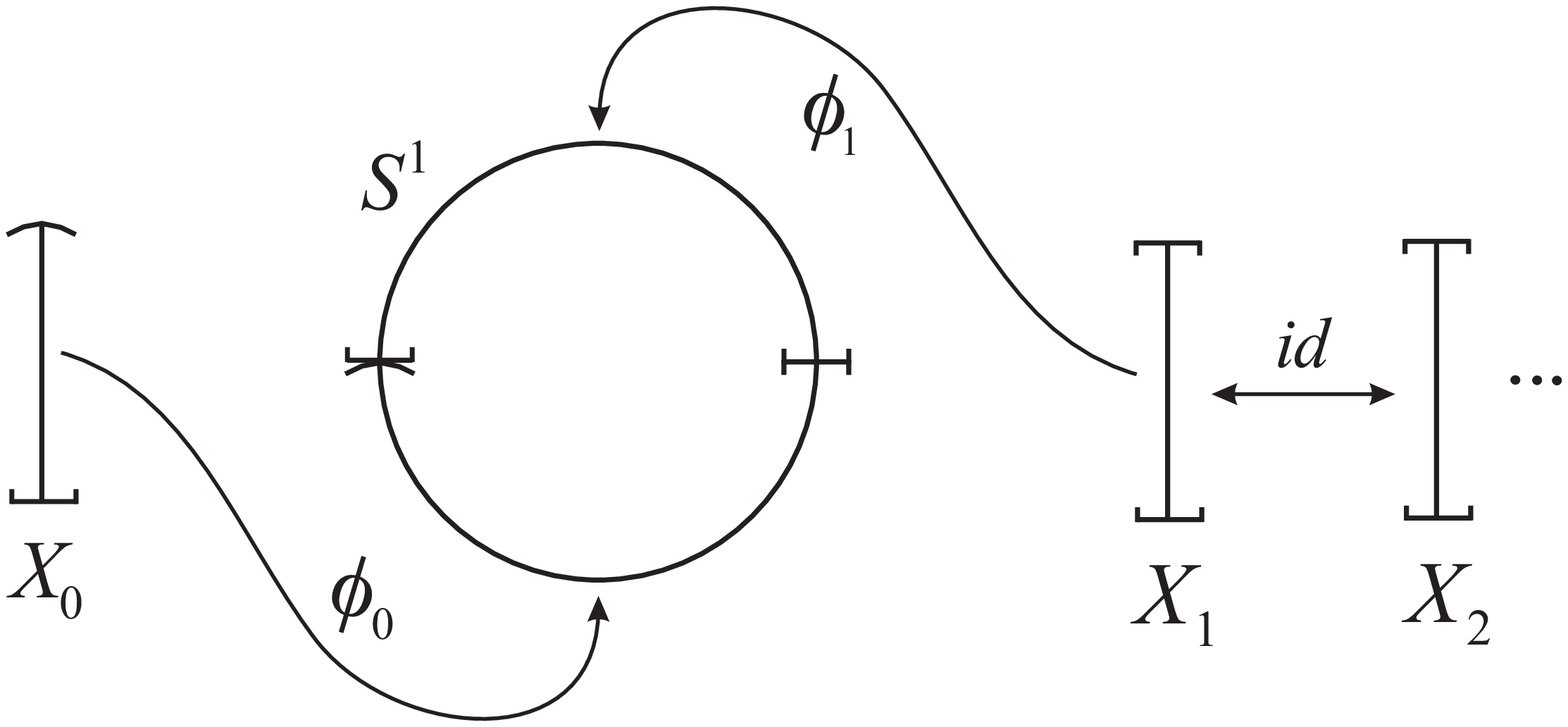}
\caption{{\small Limit space (not fundamental)}}
\end{minipage} \hfill
\begin{minipage}[b]{0.46\linewidth}
\centering
\includegraphics[scale=0.24]{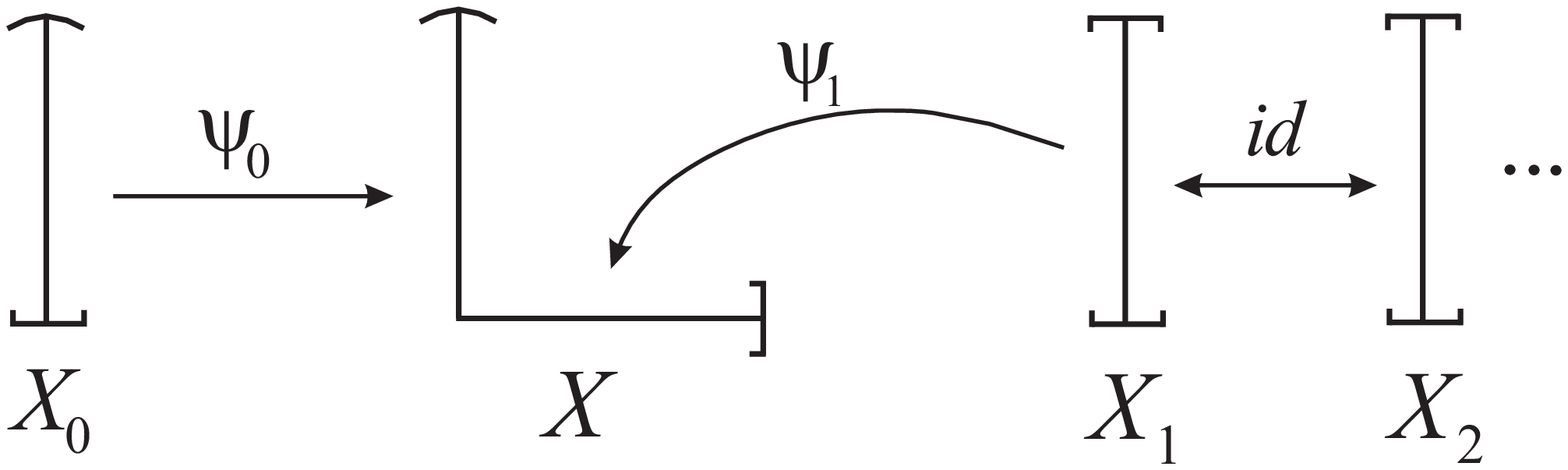}
\caption{{\small Fundamental limit space}}
\end{minipage}
\end{figure}

Now, we consider the subspace $X=\{(x,0)\in\R^2 : 0\leq
x\leq1\}\cup\{(0,y)\in\R^2 : 0\leq y<1\}$ of $\R^2$. Define
$$\psi_{0}:X_{0}\rightarrow X, \ {\rm by} \ \psi_{0}(t)=(0,t) \
{\rm and}$$
$$\psi_{i}:X_{i}\rightarrow X, \ {\rm by} \ \psi_{i}(t)=(t,0), \ {\rm for \ all} \
i\geq1.$$

We have $X=\bigcup_{i=0}^{\infty}\psi_{i}(X_{i})$, where each
$\phi_{i}$ is an imbedding onto its image, such that
$\psi_{i}(X_{i})$ is closed in $X$. Moreover, since
$\psi_{i}(X_{i})\cap\psi_{j}(X_{j})\doteq\psi_{j}f_{i,j-1}(Y_{i})$,
it follows that $\{X,\psi_{i}\}$ is a fundamental limit space for
the CIS $\SIF$. (See Figure 2 above).

(The bijection $\beta:S^{1}\rightarrow X$ of the Theorem
\ref{Teo-Bijecao} is not continuous here).

\end{Exmp}

\begin{Exmp} \label{Exemplo.S.infinito}
The infinite-dimensional sphere $S^{\infty}$.

For each $n\in\N$, we consider the $n$-dimensional sphere
$$S^n=\{(x_{1},\ldots,x_{n+1})\in\R^{n+1}:x_{1}^2+\cdots+x_{n+1}^2=1\},$$
and the ``equatorial inclusions'' $f_{n}:S^n\rightarrow S^{n+1}$
given by
$$f_{n}(x_{1},\ldots,x_{n+1})=(x_{1},\ldots,x_{n+1},0).$$ Then
$\{S^n,S^n,f_{n}\}$ is a CIS. Its fundamental limit space is
$\{S^{\infty},\phi_{n}\}$, where $S^{\infty}$ is the {\it
infinite-dimensional sphere} and, for each $n\in\N$, the imbedding
$\phi_{n}:S^n\rightarrow S^{\infty}$ is the natural ``equatorial
inclusion''.

\end{Exmp}

\begin{Exmp}
The infinite-dimensional torus $T^{\infty}$.

For each $n\geq1$, we consider the $n$-dimensional torus
$T^n=\prod_{i=1}^{n}S^1$ and the closed injections
$f_{n}:T^n\rightarrow T^{n+1}$ given by
$f_{n}(x_{1},\ldots,x_{n})=(x_{1},\ldots,x_{n},(1,0))$, where each
$x_{i}\in S^{1}$. Then $\{T^n,T^n,f_{n}\}$ is a CIS, whose
fundamental limit space is $\{T^{\infty},\phi_{n}\}$, where
$T^{\infty}=\prod_{i=1}^{\infty}S^1$ is the {\it
infinite-dimensional torus} and, for each $n\in\N$, the imbedding
$\phi_{n}:T^n\rightarrow T^{\infty}$ is the natural inclusion
$$\phi_{n}(x_{1},\ldots,x_{n})=(x_{1},\ldots,x_{n},(1,0),(1,0),\ldots).$$

\end{Exmp}

Example \ref{Exemplo.S.infinito} is a particular case the following
one:

\begin{Exmp}\label{Exemplo.CW-complex}
The CW-complexes as fundamental limit spaces for its skeletons.

Let $K$ be an arbitrary CW-complex. For each $n\in\N$, let $K^n$ be
the $n$-skeleton of $K$ and consider the natural inclusions
$l_{n}:K^n\rightarrow K^{n+1}$ of the $n$-skeleton into
$(n+1)$-skeleton. If the dimension $\dim(K)$ of $K$ is finite, then
we put $K^m=K$ and $l_{m}:K^m\rightarrow K^{m+1}$ to be the identity
map, for all $m\geq\dim(K)$. It is known that a CW-complex has the
weak topology with respect to their skeletons, that is, a subset
$A\subset K$ is closed in $K$ if and only if $A\cap K^n$ is closed
in $K^n$ for all $n$. Thus, $\{K^n,K^n,l_{n}\}$ is a CIS, whose
fundamental limit space is $\{K,\phi_{n}\}$, where each
$\phi_{n}:K^n\rightarrow K$ is the natural inclusions of the
$n$-skeleton $K^n$ into $K$.
\end{Exmp}

For details of the CW-complex theory see \cite{Hatcher} or
\cite{Whitehead}.

The example below is a consequence of the previous one.

\begin{Exmp} \label{Exemplo.RP.infinito}
The infinite-dimensional projective space $\RP^{\infty}$.

There is always a natural inclusion
$f_{n}:\RP^{n}\rightarrow\RP^{n+1}$, which is a closed injective
continuous map. {\rm(}$\RP^{n}$ is the $n$-skeleton of the
$\RP^{n+1}${\rm )}. It follows that $\{\RP^{n},\RP^{n},f_{n}\}$ is a
CIS. The fundamental limit space for this CIS is the {\it
infinite-dimensional projective space} $\RP^{\infty}$.

\end{Exmp}

For details about infinite-dimensional sphere and projective plane
see \cite{Hatcher}.

\section{The category of closed injective systems and cis-morphisms}\label{Section.Category} 

Let $\X=\{X_{i},Y_{i},f_{i}\}_{i}$ and
$\Ze=\{Z_{i},W_{i},g_{i}\}_{i}$ be two closed injective systems. By
a {\bf cis-morphism} $\mathfrak{h}:\X\rightarrow\Ze$ we mean a
collection
$$\mathfrak{h}=\{h_{i}:X_{i}\rightarrow Z_{i}\}_{i}$$ of closed continuous maps checking the following
conditions:
\begin{enumerate}
\item[{\bf 1.}] $h_{i}(Y_{i})\subset W_{i}$, for all
$i\in\N$.
\item[{\bf 2.}] $h_{i+1}\circ f_{i}=g_{i}\circ h_{i}|_{Y_{i}}$, for all
$i\in\N$.
\end{enumerate}
This latter condition is equivalent to commutativity of the diagram
below, for each $i\in\N$.

\begin{table}[h]
\centering
\begin{tabular}{c} \xymatrix{ Y_{i} \ar[d]_{f_{i}} \ar[rr]^{h_{i}|_{Y_{i}}} & & W_{i} \ar[d]^{g_{i}}\\
X_{i+1} \ar[rr]_{h_{i+1}} & & Z_{i+1} \\ }
\end{tabular}
\end{table}

We say that a cis-morphism $\mathfrak{h}:\X\rightarrow\Ze$ is a {\bf
cis-isomorphism} if each map $h_{i}:X_{i}\rightarrow Z_{i}$ is a
homeomorphism and carries $Y_{i}$ homeomorphicaly onto $W_{i}$.

\vspace{2mm}

For each arbitrary CIS, say $\X=\{X_{i},Y_{i},f_{i}\}_{i}$, there is
an identity cis-morphism $\mathfrak{1}:\X\rightarrow\X$ given by
$\mathfrak{1}_{i}:X_{i}\rightarrow X_{i}$ equal to identity map for
each $i\in\N$.

Moreover, if $\mathfrak{h}:\X^{(1)}\rightarrow\X^{(2)}$ and
$\mathfrak{k}:\X^{(2)}\rightarrow\X^{(3)}$ are two cis-morphisms,
then it is clear that its natural composition
$$\mathfrak{k}\circ\mathfrak{h}:\X^{(1)}\rightarrow\X^{(3)}$$
is a cis-morphism from $\X^{(1)}$ into $\X^{(3)}$.

Also, it is easy to check that associativity of compositions holds
whenever possible: if $\mathfrak{h}:\X^{(1)}\rightarrow\X^{(2)}$,
$\mathfrak{k}:\X^{(2)}\rightarrow\X^{(3)}$ and
$\mathfrak{r}:\X^{(3)}\rightarrow\X^{(4)}$, then
$$\mathfrak{r}\circ(\mathfrak{k}\circ\mathfrak{h})=(\mathfrak{r}\circ\mathfrak{k})\circ\mathfrak{h}.$$

This shows that the closed injective system and the cis-morphisms
between they forms a category, which we denote by $\mathfrak{Cis}$.
(See \cite{Hotman} for details on basic category theory).

\begin{thm}\label{CIS.Category.Is.Complet}
Every inductive systems on the category $\mathfrak{Cis}$ admit
limit.
\end{thm}
\begin{proof}
Let $\{\X^{(n)},\mathfrak{h}^{(mn)}\}_{m,n}$ be an inductive system
of closed injective system and cis-morphisms. Then, each $\X^{(n)}$
is of the form
$\X^{(n)}=\{X_{i}^{(n)},Y_{i}^{(n)},f_{i}^{(n)}\}_{i}$ and each
$\mathfrak{h}^{(mn)}:\X^{(m)}\rightarrow\X^{(n)}$ is a cis-morphism
and, moreover,
$\mathfrak{h}^{(pq)}\circ\mathfrak{h}^{(qr)}=\mathfrak{h}^{(pr)}$,
for all $p,q,r\in\N$.

For each $m\in\N$, we write $\mathfrak{h}^{(m)}$ to be
$\mathfrak{h}^{(mn)}$ when $m=n+1$.

For each $i\in\N$, we have the inductive system
$\{X_{i}^{(n)},h_{i}^{(mn)}\}_{m,n}$, that is, the injective system
of the topological spaces $X_{i}^{(1)},X_{i}^{(2)},\ldots$ and all
continuous maps $h_{i}^{(mn)}:X_{i}^{(m)}\rightarrow X_{i}^{(n)}$,
$m,n\in\N$, of the collection $\mathfrak{h}^{(mn)}$.

Now, each inductive system $\{X_{i}^{(n)},h_{i}^{(mn)}\}_{m,n}$ can
be consider as the closed injective system
$\{X_{i}^{(n)},X_{i}^{(n)},h_{i}^{(n)}\}_{n}$. Let
$\{X_{i},\xi_{i}^{(n)}\}_{n}$ be a fundamental limit space for
$\{X_{i}^{(n)},X_{i}^{(n)},h_{i}^{(n)}\}_{n}$.

%

\vspace{-2mm}

\begin{table}[h]
\centering
\begin{tabular}{c} \xymatrix{ & & & & X_{i} \\
\ar[r]_-{h_{i}^{(qm)}} & X_{i}^{(m)} \ar[rr]_{h_{i}^{(mn)}}
\ar[rrru]^-{\xi_{i}^{(m)}} & & X_{i}^{(n)} \ar[r]_{h_{i}^{(np)}}
\ar[ru]_-{\xi_{i}^{(n)}} & }
\end{tabular}
\end{table}

\vspace{-2mm}

Then, each $\xi_{i}^{(n)}:X_{i}^{(n)}\rightarrow X_{i}$ is an
imbedding, and we have $\xi_{i}^{(m)}=\phi_{i}^{(n)}\circ
h_{i}^{(mn)}$ for all $m<n$. Moreover, $X_{i}$ has a weak topologia
induced by the collection $\{\xi_{i}^{(n)}\}_{n}$.

For any $m,n\in\N$, with $m\leq n$, we have
$$\xi_{i}^{(m)}(Y_{i}^{(m)})=\xi_{i}^{(n)}\circ
h_{i}^{(mn)}(Y_{i}^{(m)})\subset\xi_{i}^{(n)}(Y_{i}^{(n)}),$$ by
condition {\bf 1} of the definition of cis-morphism. Moreover, each
$\xi_{i}^{(n)}(Y_{i}^{(n)})$ is closed in $X_{i}$, since each
$\xi_{i}^{(n)}$ is an imbedding.

For each $i\in\N$, we define

\vspace{-5mm}
$$Y_{i}=\bigcup_{n\in\N}\xi_{i}^{(n)}(Y_{i}^{(n)}).$$

Then, by preceding paragraph, $Y_{i}$ is a union of linked closed
sets, that is, $Y_{i}$ is the union of the closed sets of the
ascendent chain
$$\xi_{i}^{(1)}(Y_{i}^{(1)})\subset\xi_{i}^{(2)}(Y_{i}^{(2)})\subset\cdots
\subset\xi_{i}^{(m)}(Y_{i}^{(m)})\subset\xi_{i}^{(m+1)}(Y_{i}^{(m+1)})\subset\cdots$$

Now, since  $\{X_{i},\xi_{i}^{(n)}\}_{n}$ is a fundamental limit
space for $\{X_{i}^{(n)},Y_{i}^{(n)},h_{i}^{(n)}\}_{n}$, for each
$m\in\N$, we have
$$(\xi_{i}^{(m)})^{-1}(Y_{i})=(\xi_{i}^{(m)})^{-1}(\cup_{n\in\N}\xi_{i}^{(n)}(Y_{i}^{(n)}))=Y_{i}^{m}
 \ {\rm which \ is \ closed \ in} \ X_{i}^{(m)}.$$

Therefore, since $X_{i}$ has the weak topology induced by the
collection $\{\xi_{i}^{(n)}\}_{n}$, it follows that $Y_{i}$ is
closed in $X_{i}$.

Now, we will build, for each $i\in\N$, an injection
$f_{i}:Y_{i}\rightarrow X_{i+1}$ making $\{X_{i},Y_{i},f_{i}\}_{i}$
a closed injective system. For each $i\in\N$, we have the diagram
shown below.

For each $x\in {\xi_{i}^{(n)}}(Y_{i}^{(n)})\subset X_{i}$, there is
a unique $y\in Y_{i}^{(n)}$ such that $\xi_{i}^{(n)}(y)=x$. Then, we
define $f_{i}(x)=(\xi_{i+1}^{(n)}\circ f_{i}^{(n)})(y)$.

\vspace{-4mm}

\begin{table}[h]
\centering
\begin{tabular}{c} \xymatrix{ Y_{i}^{(n)} \ar[d]_{f_{i}^{(n)}} \ar[rr]^{\xi_{i}^{(n)}} & & {\xi_{i}^{(n)}}(Y_{i}^{(n)}) \ar@{.>}[d]^{f_{i}}\\
X_{i+1}^{(n)} \ar[rr]_{\xi_{i+1}^{(n)}} & & X_{i+1} \\ }
\end{tabular}
\end{table}

\newpage

It is clear that each $f_{i}:{\xi_{i}^{(n)}}(Y_{i}^{(n)})\rightarrow
X_{i+1}$ is a closed injective continuous map, since each $\xi_{i}$
and $f_{i}^{(n)}$ are closed injective continuous maps.

Now, we define $f_{i}:Y_{i}\rightarrow X_{i+1}$ in the following
way: For each $x\in Y_{i}$, there is an integer $n\in\N$ such that
$x\in\xi_{i}^{(n)}(Y_{i}^{n})$. Then, there is a unique $y\in
Y_{i}^{(n)}$ such that $\xi_{i}^{(n)}(y)=x$. We define
$f_{i}(x)=(\xi_{i+1}^{(n)}\circ f_{i}^{(n)})(y)$.

Each $f_{i}:Y_{i}\rightarrow X_{i+1}$ is well defined. In fact:
suppose that $x$ belong to $\xi_{i}^{(m)}(Y_{i}^{m})\cap
\xi_{i}^{(n)}(Y_{i}^{n})$. Suppose, without loss of generality, that
$m<n$. There are unique $y_{m}\in Y_{i}^{m}$ and $y_{n}\in
Y_{i}^{n}$, such that
$\xi_{i}^{(m)}(y_{m})=y=\phi_{i}^{(n)}(y_{n})$. Then,
$y_{n}=h_{i}^{(mn)}(y_{m})$. Thus,
$$\xi_{i+1}^{(n)}\circ f_{i}^{(n)}(y_{n})=\xi_{i+1}^{(n)}\circ f_{i}^{(n)}\circ h_{i}^{(mn)}(y_{m})=
\xi_{i+1}^{(n)}\circ h_{i+1}^{(mn)}\circ f_{i}^{(m)}(y_{m})=
\xi_{i+1}^{(m)}\circ f_{i}^{(m)}(y_{m}).$$

Now, since each $f_{i}:Y_{i}\rightarrow X_{i+1}$ is obtained of a
collection of closed injective continuous maps which coincides on
closed sets, it follows that each $f_{i}$ is a closed injective
continuous map.

This proves that $\{X_{i},Y_{i},f_{i}\}_{i}$ is a closed injective
system. Denote it by $\X$.

For each $n\in\N$, let $\mathcal{E}^{(n)}:\X^{(n)}\rightarrow\X$ be
the collection
$$\mathcal{E}^{(n)}=\{\xi_{i}^{(n)}:X_{i}^{(n)}\rightarrow
X_{i}\}_{i}.$$ It is clear by the construction that
$\mathcal{E}^{(n)}$ is a cis-morphism from $\X^{(n)}$ into $\X$.
Moreover, we have
$\mathcal{E}^{(m)}=\mathcal{h}^{(mn)}\circ\mathcal{E}^{(n)}$.
Therefore, $\{\X,\mathcal{E}^{(n)}\}_{n}$ is a direct limit for the
inductive system $\{\X^{(n)},\mathfrak{h}^{(mn)}\}_{m,n}$.
\end{proof}

\section{The transition to fundamental limit space as a
functor}\label{Section.Passagem.ao.limite} 

Henceforth, we will write $\mathfrak{Top}$ to denote the category of
the topological spaces and continuous maps.

For each CIS $\X=\{X_{i},Y_{i},f_{i}\}_{i}$, we will denote its
fundamental limite space by $\pounds(\X)$. The passage to the
fundamental limit defines a function
$$\pounds:\mathfrak{Cis}\longrightarrow\mathfrak{Top}$$ which associates to each
CIS $\X$ its fundamental limit space $\pounds(\X)=\{X,\phi_{i}\}$.

\begin{thm}\label{Theorem.Fundamental.Induced.Map}
Let $\mathfrak{h}:\X\rightarrow\Ze$ be a cis-morphism between closed
injective systems, and let $\pounds(\X)=\{X,\phi_{i}\}_{i}$ and
$\pounds(\Ze)=\{Z,\psi_{i}\}_{i}$ be the fundamental limit spaces
for $\X$ and $\Ze$, respectively. Then, there is a unique closed
continuous map $\pounds\mathfrak{h}:X\rightarrow Z$ such that
$\pounds\mathfrak{h}\circ\phi_{i}=\psi_{i}\circ h_{i}$, for all
$i\in\N$.
\end{thm}
\begin{proof} Write $\mathfrak{h}=\{h_{i}:X_{i}\rightarrow Z_{i}\}_{i}$. We define the map
$\pounds\mathfrak{h}:X\rightarrow Z$ as follows: First, consider
$\pounds(\X)=\{X,\phi_{i}\}$ and $\pounds(\Ze)=\{Z,\psi_{i}\}$. For
each $x\in X$, there is $x_{i}\in X_{i}$, for some $i\in\N$, such
that $x=\phi_{i}(x_{i})$. Then, we define
$$\pounds\mathfrak{h}(x)=\psi_{i}\circ h_{i}(x_{i}).$$

This map is well defined. In fact, if
$x=\phi_{i}(x_{i})=\phi_{j}(x_{j})$, with $i<j$, then
$x\in\phi_{i}(X_{i})\cap\phi_{j}(X_{j})\doteq\phi_{j}f_{i,j-1}(Y_{i,j-1})$
and $x_{j}=f_{i,j-1}(x_{i})$. Thus,
$$\psi_{j}\circ h_{j}(x_{j})=\psi_{j}\circ h_{j}\circ
f_{i,j-1}(x_{i})=\psi_{j}\circ g_{i,j-1}\circ
h_{i}(x_{i})=\psi_{i}\circ h_{i}(x_{i}).$$

Now, since $\pounds\mathfrak{h}$ is obtained from a collection of
closed continuous maps which coincide on closed sets, it is easy to
see that $\pounds\mathfrak{h}$ is a closed continuous map.

Moreover, it is easy to see that $\pounds\mathfrak{h}$ is the unique
continuous map from $X$ into $Z$ which verifies, for each $i\in\N$,
the commutativity $\pounds\mathfrak{h}\circ\phi_{i}=\psi_{i}\circ
h_{i}$.
\end{proof}

Sometimes, we write
$\pounds\mathfrak{h}:\pounds(\X)\rightarrow\pounds(\Ze)$ instead
$\pounds\mathfrak{h}:X\rightarrow Y$. This map is called the {\bf
fundamental map} induced by $\mathfrak{h}$.


\begin{cor}
The transition to the fundamental limit space is a functor from the
category $\mathfrak{Cis}$ into the category $\mathfrak{Top}$.
\end{cor}

For details on functors see \cite{Hotman}.

\begin{cor}
If $\mathfrak{h}:\X\rightarrow\Ze$ is a cis-isomorphism, then the
fundamental map
$\pounds\mathfrak{h}:\pounds(\X)\rightarrow\pounds(\Ze)$ is a
homeomorphism.
\end{cor}

This implies that isomorphic closed injective systems have
homeomorphic fundamental limit spaces.

\section{Compatibility of
limits}\label{Section.Compatibility.Limits} 

In this section, given a CIS $\X=\SIF$ with fundamental limit space
$\{X,\phi_{i}\}$, sometimes we write $\pounds(\X)$ to denote only
the topological space $X$. This is clear in the context.

\begin{thm}\label{Teo.Compatibilidade.Limits}
Let $\{\X^{(n)},\mathfrak{h}^{(mn)}\}_{m,n}$ be an inductive system
on the category $\mathfrak{Cis}$ and let
$\{\X,\mathcal{E}^{(n)}\}_{n}$ its direct limit. Then
$\{\pounds(\X^{(n)}),\pounds\mathfrak{h}^{(mn)}\}_{m,n}$ is an
inductive system on the category $\mathfrak{Top}$, which admits
$\pounds(\X)$ as its directed limit homeomorphic.
\end{thm}
\begin{proof}
By uniqueness of the direct limit, we can assume that
$\{\X,\Phi^{(n)}\}_{n}$ is the direct limit constructed in the proof
of the Theorem \ref{CIS.Category.Is.Complet}. Then, we have
$$\mathcal{E}^{(n)}:\X^{(n)}\rightarrow\X \ \ {\rm given \ by} \ \ \mathcal{E}^{(n)}=\{\xi_{i}^{(n)}:X_{i}^{(n)}\rightarrow X_{i}\}_{i},$$
where $\{X_{i},\xi_{i}^{(n)}\}_{n}$ is a fundamental limit space for
$\{X_{i}^{(n)},X_{i}^{(n)},h_{i}^{(n)}\}_{n}$.

\vspace{1mm}

By the Theorem \ref{Theorem.Fundamental.Induced.Map},
$\{\pounds(\X^{(n)}),\pounds\mathfrak{h}^{(mn)}\}_{m,n}$ is a
inductive system.

For each $n\in\N$, write
$\X^{(n)}=\{X_{i}^{(n)},Y_{i}^{(n)},f_{i}^{(n)}\}_{i}$ and
$\pounds(\X^{n})=\{X^{(n)},\phi_{i}^{(n)}\}_{i}$. Moreover, write
$\X=\{X_{i},Y_{i},f_{i}\}_{i}$ and $\pounds(\X)=\{X,\phi_{i}\}_{i}$.
Then, the inductive system
$\{\pounds(\X^{(n)}),\pounds\mathfrak{h}^{(mn)}\}_{m,n}$ can be
write as $\{X^{(n)},\pounds\mathfrak{h}^{(mn)}\}_{m,n}$.

We need to show that there is a collection of maps
$\{\vartheta^{(n)}:X^{(n)}\rightarrow X\}_{n}$ such that
$\{X,\vartheta^{(n)}\}_{n}$ is a direct limit for the system
$\{X^{(n)},\pounds\mathfrak{h}^{(mn)}\}_{m,n}$.

For each $x\in X^{(n)}$, there is a point $x_{i}\in X_{i}^{(n)}$,
for some $i\in\N$, such that $x=\phi_{i}^{(n)}(x_{i})$. We define
$\vartheta^{(n)}:X^{(n)}\rightarrow X$ by
$\vartheta^{(n)}(x)=\phi_{i}\circ\xi_{i}^{(n)}(x_{i})$.

%

The map $\vartheta^{(n)}$ is well defined. In fact: If
$x=\phi_{i}^{(n)}(x_{i})=\phi_{j}^{(n)}(x_{j})$, with $i\leq j$,
then
$x\in\phi_{i}^{(n)}(X_{i}^{(n)})\cap\phi_{j}^{(n)}(X_{j}^{(n)})\doteq\phi_{j}^{(n)}f_{i,j-1}^{(n)}(Y_{i,j-1}^{(n)})$
and $x_{j}=f_{i,j-1}^{(n)}(x_{i})$ and $x_{i}\in Y_{i,j}\subset
X_{i}$. Now, in the diagram below, the two triangles and the big
square are commutative. In it, we write $\xi_{i}^{(n)}|$ and
$\phi_{i}^{(n)}|$ to denote the obvious restriction.

It follows that
$$\phi_{j}\circ\xi_{j}^{(n)}(x_{j})=\phi_{j}\circ\xi_{j}^{(n)}\circ f_{i,j-1}{(n)}(x_{i})=\phi_{j}\circ f_{i,j-1}\circ\xi_{i}^{(n)}(x_{i})=
\phi_{i}\circ\xi_{i}^{(n)}(x_{i}).$$ It is sufficient to prove that
the map $\vartheta^{(n)}$ is well defined. Moreover, note that this
map makes the diagram above in a commutative diagram.

\newpage

\begin{table}[h]
\centering
\begin{tabular}{c} \xymatrix{ X_{i} \ar[rr]^{f_{i,j-1}} \ar[rd]_{\phi_{i}} & & X_{j} \ar[ld]^{\phi_{j}} \\
 & X & \\
 & X^{(n)} \ar@{.>}[u]_{\vartheta^{(n)}} & \\
 Y_{i,j}^{(n)} \ar[uuu]^{\xi_{i}^{(n)}|} \ar[ru]^{\phi_{i}^{(n)}|}
 \ar[rr]_{f_{i,j-1}^{(n)}} & & X_{j}^{(n)} \ar[lu]_{\phi_{j}^{(n)}}
 \ar[uuu]_{\xi_{j}^{(n)}} \\ }
\end{tabular}
\end{table}


Now, by the Theorem \ref{Theorem.Fundamental.Induced.Map} we have
$\pounds\mathfrak{h}^{(mn)}\circ\phi_{i}^{(m)}=\phi_{i}^{(n)}\circ
h_{i}^{(n)}$ for all integers $m<n$, since
$\pounds(\X^{n})=\{X^{(n)},\phi_{i}^{(n)}\}_{i}$.

Let $x\in X^{(m)}$ be an arbitrary point. Then, there is $x_{i}\in
X_{i}^{(m)}$ such that $x=\phi_{i}^{(m)}(x_{i})$. Also, for all
$n\in\N$ with $m<n$, we have
$\pounds\mathfrak{h}^{(mn)}(x)=\phi_{i}^{(n)}\circ
h_{i}^{(mn)}(x_{i})$. Thus, we have,
$$\vartheta^{(n)}\circ\pounds\mathfrak{h}^{(mn)}(x)=\phi_{i}\circ\xi_{i}^{(n)}(h_{i}^{(mn)}(x_{i}))=\phi_{i}\circ\xi^{(m)}(x_{i})=\vartheta^{(m)}(x).$$

This shows that
$\vartheta^{(n)}\circ\pounds\mathfrak{h}^{(mn)}=\vartheta^{(m)}$ for
all integers $m<n$.

Let $A$ be a closed subset of $X$. Then it is clear that
$(\phi_{i}\circ\xi_{i}^{(n)})^{-1}(A)$ is closed in $X_{i}^{(n)}$,
since $\phi_{i}$ and $\xi_{i}^{(n)}$ are continuous maps. Now, we
have
$(\vartheta^{(n)})^{-1}(A)=\phi_{i}^{(n)}((\phi_{i}\circ\xi_{i}^{(n)})^{-1}(A))$.
Then, since $\phi_{i}^{(n)}$ is an imbedding (and so a closed map),
it follows that $(\vartheta^{(n)})^{-1}(A)$ is a closed subset of
$X^{(n)}$. Therefore, $\vartheta^{(n)}$ is a continuous.

Now, it is not difficult to prove that $\{X,\vartheta^{(n)}\}_{n}$
satisfies the universal mapping problem (see \cite{Hotman}). This
concludes the proof.
\end{proof}

\section{Inductive closed injective
systems}\label{Section.Inductive.System}

In this section, we will study a particular kind of closed injective
systems, which has some interesting properties. More specifically,
we study the CIS's of the form $\{X_{i},X_{i},f_{i}\}$, which are
called {\bf inductive closed injective system}, or an inductive CIS,
to shorten.

In an inductive CIS $\SIFI$, any two injections $f_{i}$ and $f_{j}$,
with $i<j$, are {\bf componible}, that is, the composition
$f_{i,j}=f_{j}\circ\cdots\circ f_{i}$ is always defined throughout
domine $X_{i}$ of $f_{i}$.

Hence, fixing  $i\in\N$, for each $j>i$ we have a closed injection
$f_{i,j}:X_{i}\rightarrow X_{j+1}$. Because this, we define, for
each $i<j\in\N$,
$$f_{i}^{i}=id_{X_{i}}:X_{i}\rightarrow X_{i}$$
$$f_{i}^{j}=f_{i,j-1}:X_{i}\rightarrow  X_{j}.$$

By this definition, it follows that $f_{i}^{k}=f_{j}^{k}\circ
f_{i}^{j}$, for all $i\leq j\leq k$. Therefore,
$\{X_{i},f_{i}^{j}\}$ is an {\bf inductive system} on the category
$\mathfrak{Top}$.

We will construct a direct limit for this inductive system.

Let $\coprod X_{i}=\coprod_{i=0}^{\infty}X_{i}$ be the coproduct (or
topological sum) of the spaces $X_{i}$.

Consider the canonical inclusions
$\varphi_{i}:X_{i}\rightarrow\coprod X_{i}$. It is obvious that each
$\varphi_{i}$ is a homeomorphism onto its image.

Over the space $\coprod X_{i}$, consider the relation $\sim$ defined
by:
$$x\sim y\Leftrightarrow \left\{\begin{array}{l} \exists \
x_{i}\in X_{i}, y_{j}\in X_{j} \ {\rm with} \ x=\varphi_{i}(x_{i}) \
{\rm e} \ y=\varphi_{j}(y_{j}), \ {\rm such \ that} \\
y_{j}=f_{i}^{j}(x_{i}) \ {\rm if} \ i\leq j \ {\rm and} \
x_{i}=f_{j}^{i}(y_{j}) \ {\rm if} \ j<i. \\ \end{array} \right..$$

\begin{lem}\label{Lema.Relação.equivalencia}
The relation $\sim$ is an equivalence relation over $\coprod X_{i}$.
\end{lem}
\begin{proof}
We need to check the veracity of the properties reflexive, symmetric and transitive.

{\it Reflexive}: Let $x\in X$ be a point. There is $x_{i}\in X_{i}$
such that $x=\psi_{i}(x_{i})$, for some $i\in\N$. We have
$x_{i}=f_{i}^{i}(x_{i})$. Therefore $x\sim x$.

{\it Symmetric}: It is obvious by definition of the relation $\sim$.

{\it Transitive}: Assume that $x\sim y$ and $y\sim z$. Suppose that
$x=\varphi_{i}(x_{i})$ and $y=\varphi_{j}(y_{j})$ with
$y_{j}=f_{i}^{j}(x_{i})$. In this case, $i\leq j$. (The other case
is analogous and is omitted). Since $y\sim z$, we can have:

{\it Case 1 }: $y=\varphi_{j}(y_{j}')$ and $z=\varphi_{k}(z_{k})$
with $j\leq k$ and $z_{k}=f_{j}^{k}(y_{j}')$. Then
$\varphi_{j}(y_{j})=y=\varphi_{j}(y_{j}')$, and so $y_{j}=y_{j}'$.
Since $i\leq j\leq k$, we have
$z_{k}=f_{j}^{k}(y_{j})=f_{j}^{k}f_{i}^{j}(x_{i})=f_{i}^{k}(x_{i})$.
Therefore $x\sim z$.

{\it Case 2}: $y=\varphi_{j}(y_{j}')$ and $z=\varphi_{k}(z_{k})$
with $k<j$ and $y_{j}'=f_{k}^{j}(z_{k})$. Then $y_{j}=y_{j}'$, as
before. Now, we have again two possibility:

($a$) If $i\leq k<j$, then we have
$f_{k}^{j}(z_{k})=y_{j}=f_{i}^{j}(x_{i})=f_{k}^{j}f_{i}^{k}(x_{i})$.
Thus $z_{k}=f_{i}^{k}(x_{i})$ and $x\sim z$.

($b$) If $k< i \leq j$, then we have
$f_{i}^{j}(x_{i})=y_{j}=f_{k}^{j}(z_{k})=f_{i}^{j}f_{k}^{i}(z_{k})$.
Thus $x_{i}=f_{k}^{i}(z_{k})$ and $x\sim z$.
\end{proof}

Let $\widetilde{X}=(\coprod X_{i}) / \sim$ be the quotient space
obtained of $\coprod X_{i}$ by the equivalence relation $\sim$, and
for each $i\in\N$, let
$\widetilde{\varphi_{i}}:X_{i}\rightarrow\widetilde{X}$ be the
composition $\widetilde{\varphi_{i}}=\rho\circ\varphi_{i}$, where
$\rho:\coprod X_{i}\rightarrow\widetilde{X}$ is the quotient
projection.
$$ \xymatrix{\widetilde{\varphi_{i}}: X_{i} \ \ar[r]^{\varphi_{i}} &
\coprod X_{i} \ar[r]^{\rho} & \widetilde{X}}$$

Note that, since $\widetilde{X}$ has the quotient topology induced
by projection $\rho$, a subset $A\subset\widetilde{X}$ is closed in
$\widetilde{X}$ if and only if $\widetilde{\varphi_{i}}^{-1}(A)$ is
close in $X_{i}$, for each $i\in\N$.

Given $x,y\in\coprod X_{i}$ with $x,y\in X_{i}$, then $x\sim
y\Leftrightarrow x=y$. Thus, each $\widetilde{\varphi_{i}}$ is
one-to-one fashion onto $\widetilde{\varphi_{i}}(X_{i})$. Moreover,
it is obvious that
$\widetilde{X}=\cup_{i=0}^{\infty}\widetilde{\varphi_{i}}(X_{i})$.

These observations suffice to conclude the following:

\begin{thm} \label{EL==LD}
$\{\widetilde{X},\widetilde{\varphi_{i}}\}$ is a fundamental limit
space for the inductive CIS $\SIFI$. Moreover,
$\{\widetilde{X},\widetilde{\varphi_{i}}\}$ is a direct limit for
the inductive system $\{X_{i},f_{i}^{j}\}$.
\end{thm}

For details on direct limit see \cite{Hotman}.

\begin{Rem} \label{Observacao.X.tilde}
If we consider an arbitrary CIS $\SIF$, then the relation $\sim$ is
again an equivalence relation over the coproduct $\coprod X_{i}$.
Moreover, in this circumstances, if $\varphi_{i}(x_{i})=x\sim
y=\varphi_{j}(y_{j})$, then we must have:
\begin{enumerate}
\item[{\rm (a)}] If $i=j$, then $x=y$.

\item[{\rm (b)}] If $i<j$, then $f_{i}$ and $f_{j-1}$ are semicomponible and
$x_{i}\in Y_{i,j-1}$;

\item[{\rm (c)}] If $i>j$, then $f_{j}$ and $f_{i-1}$ are semicomponible and
$y_{j}\in Y_{j,i-1}$.
\end{enumerate}

\noindent Therefore, it follows that the space
$\widetilde{X}=(\coprod X_{i})/\sim$ is exactly the attaching space
$X_{0}\cup_{f_{0}}X_{1}\cup_{f_{1}}X_{2}\cup_{f_{2}}\cdots$, and the
maps $\widetilde{\varphi}_{i}$ are the projections of $X_{i}$ into
this space, as in theorem of the existence of fundamental limit
space (Theorem \ref{Teo.Existencia}).
\end{Rem}

\section{Functoriality on fundamental limit spaces}\label{Section.Functor} 

Let $\funtor:\mathfrak{Top}\rightarrow \mathfrak{Mod}$ be a functor
of the category $\mathfrak{Top}$ into the category $\mathfrak{Mod}$
(the category of the $R$-modules and $R$-homomorphisms), where $R$
is a commutative ring with identity element. (See \cite{Hotman}).

Let $\SIFI$ be an arbitrary inductive CIS, and consider the
inductive system $\{X_{i},f_{i}^{j}\}$ constructed in the previous
section. The functor $\funtor$ turns this system into the inductive
system $\{\funtor X_{i},\funtor f_{i}^{j}\}$ on the category
$\mathfrak{Mod}$.

\begin{thm}\label{Teorema.Invariancia.Funtorusal} {\sc (of the Functorial Invariance)}
Let $\{X,\phi_{i}\}$ be a fundamental limit space for the inductive
CIS $\SIFI$ and let $\{M,\psi_{i}\}$ be a direct limit for
$\{\funtor X_{i},\funtor f_{i}^{j}\}$. Then, there is a unique
$R$-isomorphism $h:\funtor X\rightarrow M$ such that
$\psi_{i}=h\circ\funtor\phi_{i}$, for all $i\in\N$.
\end{thm}
\begin{proof} By the Theorem \ref{EL==LD} and by uniqueness of fundamental limit space,
there is a unique homeomorphism $\beta:X\rightarrow\widetilde{X}$
such that $\widetilde{\varphi_{i}}=\beta\circ\phi_{i}$, for all
$i\in\N$. Hence, $\funtor \beta:\funtor
X\rightarrow\funtor\widetilde{X}$ is the unique $R$-isomorphism such
that
$\funtor\widetilde{\varphi_{i}}=\funtor\beta\circ\funtor\phi_{i}$.

Since $\{\widetilde{X},\widetilde{\varphi_{i}}\}$ is a direct limit
for the inductive system $\{X_{i},f_{i}^{j}\}$ on the category
$\mathfrak{Top}$, it follows that
$\{\funtor\widetilde{X},\funtor\varphi_{i}\}$ is a direct limit of
the system $\{\funtor X_{i},\funtor f_{i}^{j}\}$ on the category
$\mathfrak{Mod}$. By universal property of direct limit, there is a
unique $R$-isomorphism $\omega:\funtor\widetilde{X}\rightarrow M$
such that $\psi_{i}=\omega\circ\funtor\widetilde{\varphi_{i}}$.

Then, we take $h:\funtor X\rightarrow M$ to be the composition
$h=\omega\circ\funtor\beta$.
\end{proof}

The universal property of direct limits among others properties can
be found, for example, in Chapter 2 of \cite{Hotman}.

Now, we describe some basic applications of the Theorem of the
Functorial Invariance.

\begin{Exmp}\label{Exemplo.Funtor.CW-complexo}
Let $K$ be an arbitrary CW-complex and let $\{K^n,K^n,l_{n}\}$ be
the CIS as in Example \ref{Exemplo.CW-complex}. It is clear that
this CIS is an inductive CIS. Let
$\funtor:\mathfrak{Top}\rightarrow\mathfrak{Mod}$ be an arbitrary
functor. Given $m<n$ in $\N$, write $l_{m}^n$ to denote the
composition $l_{n-1}\circ\cdots\circ l_{m}:K^m\rightarrow K^{n}$.
Then, $\{\funtor K^n,\funtor l_{m}^n\}$ is an inductive system on
the category $\mathfrak{Mod}$. By Theorem
\ref{Teorema.Invariancia.Funtorusal}, its direct limit is isomorphic
to $\funtor K$.

\end{Exmp}

\begin{Exmp}\label{Exemplo.Homology.S.infinito}
Homology of the sphere $S^{\infty}$.

Let $\{S^{n},S^{n},f_{n}\}$ be the CIS of Example
\ref{Exemplo.S.infinito}. Its fundamental limit space is the
infinite-dimensional sphere $S^{\infty}$. Let $p>0$ be an arbitrary
integer. By previous example, $H_{p}(S^{\infty})$ is isomorphic to
direct limit of inductive system $\{H_{p}(S^n),H_{p}(f_{m}^n)\}$,
where $f_{m}^n=f_{n-1}\circ\cdots\circ f_{m}:S^m\rightarrow S^n$,
for $m\leq n$. Now, since $H_{p}(S^n)=0$ for $n>p$, it follows that
$H_{p}(S^{\infty})=0$, for all $p>0$.

\end{Exmp}

Details on homology theory can be found in \cite{Greenberg},
\cite{Hatcher} and \cite{Spanier}.

\begin{Exmp}
The infinite projective space $\RP^{\infty}$ is a $K(\Z_{2},1)$
space.

We know that $\pi_{1}(\RP^n)\approx\Z_{2}$ for all $n\geq2$ and
$\pi_{1}(\RP^1)\approx\Z$. Moreover, for integers $m<n$, the natural
inclusion $f_{m}^n:\RP^m\hookrightarrow\RP^n$ induces a isomorphism
$(f_{m}^n)_{\#}:\pi_{1}(\RP^m)\approx\pi_{1}(\RP^n)$. For details
see [2].

The fundamental limit space for the CIS $\{\RP^{n},\RP^{n},f_{n}\}$
of the Example \ref{Exemplo.RP.infinito} is the infinite projective
space $\RP^{\infty}$. By Example \ref{Exemplo.Funtor.CW-complexo},
we have that $\pi_{1}(\RP^{\infty})$ is isomorphic to direct limit
for the inductive system $\{\pi_{1}(\RP^{n}),(f_{m}^n)_{\#}\}$.
Then, by previous arguments it is easy to check that
$\pi_{1}(\RP^{\infty})\approx\Z_{2}$.

On the other hand, for all $r>1$, we have
$\pi_{r}(\RP^{n})\approx\pi_{r}(S^n)$ for all $n\in\N$ (see [2]).
Then, $\pi_{r}(S^n)=0$ always that $1<r<n$. Thus, it is easy to
check that $\pi_{r}(\RP^{\infty})=0$, for all $r>1$.

\end{Exmp}

For details on homotopy theory and $K(\pi,1)$ spaces see
\cite{Hatcher} and \cite{Whitehead}.

\begin{Exmp}
The homotopy groups of $S^{\infty}$.

Since $\pi_{r}(S^n)=0$ for all integers $r<n$, it is very easy to
prove that $\pi_{r}(S^{\infty})=0$, for all $r\geq1$.

\end{Exmp}

\begin{Exmp}\label{Exemplo.Homlogia.Toro.infinito}
The homology of the torus $T^{\infty}$.

Some arguments very simple and similar to above can be used to prove
that $H_{0}(T^{\infty})\approx R$ and
$H_{p}(T^{\infty})\approx\bigoplus_{i=1}^{\infty}R$, for all integer
$p>0$.

\end{Exmp}

\section{Counter-Funtoriality on fundamental limit spaces}\label{Section.Counter-Functor} 

Let $\contrafuntor:\mathfrak{Top}\rightarrow\mathfrak{Mod}$ be a
counter-functor from the category $\mathfrak{Top}$ into the category
$\mathfrak{Mod}$, where $R$ is a commutative ring with identity
element.

Let $\SIFI$ be an arbitrary inductive CIS and consider the inductive
system $\{X_{i},f_{i}^{j}\}$ as before. The counter-functor
$\contrafuntor$ turns this inductive system on the category
$\mathfrak{Top}$ into the reverse system $\{\contrafuntor
X_{i},\contrafuntor f_{i}^{j}\}$ on the category $\mathfrak{Mod}$.

\begin{thm}\label{Teorema.Incariancia.Contrafuntorusal} {\sc (of the Counter-Functorial Invariance)}
Let $\{X,\phi_{i}\}$ be a fundamental limit space for the inductive
CIS $\SIFI$ and let $\{M,\psi_{i}\}$ be an inverse limit for
$\{\contrafuntor X_{i},\contrafuntor f_{i}^{j}\}$. Then there is a
unique $R$-isomorphism $h:M\rightarrow\contrafuntor X$ such that
$\psi_{i}=\contrafuntor\phi_{i}\circ h$, for all $i\in\N$.
\end{thm}
\begin{proof} By the Theorem \ref{EL==LD} and by uniqueness of fundamental limit space,
there is a unique homeomorohism $\beta:X\rightarrow\widetilde{X}$
such that $\widetilde{\varphi_{i}}=\beta\circ\phi_{i}$, for all
$i\in\N$. Hence, $\contrafuntor
\beta:\contrafuntor\widetilde{X}\rightarrow\contrafuntor X$ is the
unique $R$-isomorphism such that
$\contrafuntor\widetilde{\varphi_{i}}=\contrafuntor\phi_{i}\circ\contrafuntor\beta$.

Since $\{\widetilde{X},\widetilde{\varphi_{i}}\}$ is a direct limit
for the inductive system $\{X_{i},f_{i}^{j}\}$ on the category
$\mathfrak{Top}$, it follows that
$\{\contrafuntor\widetilde{X},\contrafuntor\varphi_{i}\}$ is an
inverse limit for the inverse system $\{\contrafuntor
X_{i},\contrafuntor f_{i}^{j}\}$ on the category $\mathfrak{Mod}$.
By universal property of inverse limit, there is a unique
$R$-isomorphism $\omega:M\rightarrow\contrafuntor\widetilde{X}$ such
that $\psi_{i}=\contrafuntor\widetilde{\varphi_{i}}\circ\omega$.

Then, we take $h:M\rightarrow\contrafuntor X$ to be the compost
$R$-isomorphism $h=\contrafuntor\beta\circ\omega$.
\end{proof}

The property of the inverse limit can be found in \cite{Hotman}.

Now, we describe some basic applications of the Theorem of the
Counter-Functorial Invariance.

\begin{Exmp}
Cohomology of the sphere $S^{\infty}$.

Since $H^p(S^n;R)\approx H_{p}(S^n;R)$ for all $p,n\in\Z$, it
follows by the Theorem \ref{Teorema.Incariancia.Contrafuntorusal}
and Example \ref{Exemplo.Homology.S.infinito} that
$H^{0}(S^{\infty};R)\approx R$ and $H^{p}(S^{\infty};R)=0$, for all
$p>0$.

\end{Exmp}

\begin{Exmp}
The cohomology of the torus $T^{\infty}$.

Since the homology and cohomology modules of a finite product of
spheres are isomorphic, it follows by Theorem the
\ref{Teorema.Incariancia.Contrafuntorusal} and Example
\ref{Exemplo.Homlogia.Toro.infinito} that $H^{0}(T^{\infty})\approx
R$ and $H^{p}(T^{\infty})\approx\bigoplus_{i=1}^{\infty}R$, for all
$p>0$.

\end{Exmp}

\section{Finitely semicomponible and stationary
CIS's}\label{Section.Stationary}

\hspace{5mm} We say that a CIS $\SIF$ is {\bf finitely
semicomponible} if, for all $i\in\N$, there is only a finite number
of indices $j\in\N$ such that $f_{i}$ and $f_{j}$ (or $f_{j}$ and
$f_{i}$) are semicomponible, that is, there is not an infinity
sequence $\{f_{k}\}_{k\geq i_{0}}$ of semicomponible maps.
Obviously, $\SIF$ is finitely semicomponible if and only if for some
(so for all) limit space $\{X,\phi_{i}\}$ for $\SIF$, the collection
$\{\phi_{i}(X_{i})\}_{i}$ is a pointwise finite cover of $X$ (that
is, each point of $X$ belongs to only a finite number of
$\phi_{i}(X_{i})'s$).

We say that a CIS $\SIF$ is {\bf stationary} if there is a
nonegative integer $n_{0}$ such that, for all $n\geq n_{0}$, we have
$Y_{n}=Y_{n_{0}}=X_{n_{0}}=X_{n}$ and $f_{n}=identity \ map$.

This section of text is devoted to the study and characterization of
the limit space of these two special types of CIS's.

\begin{thm}\label{Teorema.Semicomponivel.Fundamental}
Let $\{X,\phi_{i}\}$ be an arbitrary limit space for the CIS $\SIF$.
If the collection $\{\phi_{i}(X_{i})\}_{i}$ is a locally finite
cover of $X$, then $\SIF$ is finitely semicomponible. The reciprocal
is true if $\{X,\phi_{i}\}$ is a fundamental limit space.
\end{thm}
\begin{proof}
The first part is trivial, since if the collection
$\{\phi_{i}(X_{i})\}_{i}$ is a locally finite cover of $X$, then it
is a pointwise finite cover of $X$.

Suppose that $\{X,\phi_{i}\}$ is a fundamental limit space for the
finitely semicomponible CIS $\SIF$. Let $x\in X$ be an arbitrary
point. Then, there are nonnegative integers $n_{0}\leq n_{1}$ such
that $\phi_{i}^{-1}(\{x\})\neq\emptyset\Leftrightarrow n_{0}\leq
i\leq n_{1}$. For each $n_{0}\leq i\leq n_{1}$, write $x_{i}$ to be
the single point of $X_{i}$ such that $x=\phi_{i}(x_{i})$. It
follows that $x_{i}\in Y_{n_{i}}$ for $n_{0}\leq i\leq n_{1}-1$, but
$x_{n_{1}}\notin Y_{n_{1}}$ and $x_{n_{0}}\notin
f_{n_{0}-1}(Y_{n_{0}-1})$.

Since $f_{n_{0}-1}(Y_{n_{0}-1})$ is closed in $X_{n_{0}}$ and
$x_{n_{0}}\notin f_{n_{0}-1}(Y_{n_{0}-1})$, we can choose an open
neighborhood $V_{n_{0}}$ of $x_{n_{0}}$ in $X_{n_{0}}$ such that
$V_{n_{0}}\cap f_{n_{0}-1}(Y_{n_{0}-1})=\emptyset$.

Similarly, since $x_{n_{1}}\notin Y_{n_{1}}$ and $Y_{n_{1}}$ is
closed in $X_{n_{1}}$, we can choose an open neighborhood
$V_{n_{1}}$ of $x_{n_{1}}$ in $X_{n_{1}}$ such that $V_{n_{1}}\cap
Y_{n+{1}}=\emptyset$.

Define
$V=\phi_{n_{0}}(V_{n_{0}})\cup\phi_{n_{0}+1}(X_{n_{0}+1})\cup\cdots\cup\phi_{n_{1}-1}(X_{n_{1}-1})\cup\phi_{n_{1}}(V_{n_{1}})$.

It is clear that $x\in V\subset X$ and $V\cap
\phi_{j}(X_{j})=\emptyset$ for all $j\notin\{n_{0},\ldots,n_{1}\}$.
Moreover, we have
$$\phi_{j}^{-1}(X-V)=
\left\{\begin{array}{ccl} X_{n_{0}}-V_{n_{0}} & \mbox{if} & j=n_{0} \\
X_{n_{1}}-V_{n_{1}} & \mbox{if} & j=n_{1} \\
\emptyset & \mbox{if} & n_{0}<j<n_{1} \\
X_{j} &  & \hspace{-7mm} \mbox{otherwise} \\
\end{array} \right..$$

In all cases, we see that $\phi_{j}^{-1}(X-V)$ is closed in $X_{j}$.
Thus, $X-V$ is closed in $X$. Therefore, we obtain an open
neighborhood $V$ of $x$ which intersects only a finite number of
$\phi_{i}(X_{i})'s$.
\end{proof}

The reciprocal of the previous proposition is not true, in general,
when $\{X,\phi_{i}\} $ is not a fundamental limit space. In fact, we
have the following example in which the above reciprocal failure.

\begin{Exmp}
Consider the topological subspaces $X_{0}=[1,2]$ and
$X_{n}=[\frac{1}{n+1},\frac{1}{n}]$, for $n\geq1$, of the real line
$\R$, and take $Y_{0}=\{1\}$ and $Y_{n}=\{\frac{1}{n+1}\}$ for
$n\geq1$. Define $f_{n}:Y_{n}\rightarrow X_{n+1}$ to be the natural
inclusion, for all $n\in\N$. It is clear that the CIS
$\{X_{n},Y_{n},f_{n}\}$ is finitely semicomponible, and its
fundamental limit space is, up to homeomorphism, the subspace
$X=(0,2]$ of the real line, together the collection of natural
inclusions $\phi_{n}:X_{n}\rightarrow X$. It is also obvious that
the collection $\{\phi_{i}(X_{i})\}_{i}$ is a locally finite cover
of $X$. On the other hand, take
$$Z=((0,1]\times\{0\})\cup\{(1+cos(\pi t-\pi),\sin(\pi t-\pi))\in\R^2
: t\in[1,2]\}.$$ Consider $Z$ as a subspace of the $\R^2$. Then $Z$
is homeomorphic to the sphere $S^1$. Consider the maps
$\psi_{0}:X_{0}\rightarrow Z$ given by $\psi_{0}(t)=(1+cos(\pi
t-\pi),\sin(\pi t-\pi))$, and $\psi_{n}:X_{n}\rightarrow Z$ given by
$\psi_{n}(t)=(t,0)$, for all $n\geq1$. It is easy to see that
$\{Z,\psi_{n}\}$ is a limit space for the CIS
$\{X_{n},Y_{n},f_{n}\}$. Now, note that the point $(0,0)\in Z$ has
no open neighborhood intercepting only a finite number of
$\psi_{n}(X_{n})'s$.

\end{Exmp}

\begin{thm}\label{Teorema.Cobert.Local.Finit.Fechada}
Let $\{X,\phi_{i}\}$ be a limit space for the CIS $\SIF$ and suppose
that the collection $\{\phi_{i}(X_{i})\}_{i}$ is a locally finite
closed cover of $X$. Then $\{X,\phi_{i}\}$ is a fundamental limit
space.
\end{thm}
\begin{proof}
We need to prove that a subset $A$ of $X$ is closed in $X$ if and
only if $\phi_{i}^{-1}(A)$ is closed in $X_{i}$ for all $i\in N$.

If $A\subset X$ is closed in $X$, then it is clear that
$\phi_{i}^{-1}(A)$ is closed in $X_{i}$ for each $i\in\N$, since
each $\phi_{i}$ is a continuous map.

Now, let $A$ be a subset of $X$ such that $\phi_{i}^{-1}(A)$ is
closed in $X_{i}$, for all $i\in\N$. Then, since each $\phi_{i}$ is
a imbedding, it follows that
$\phi_{i}(\phi_{i}^{-1}(A))=A\cap\phi_{i}(X_{i})$ is closed in
$\phi_{i}(X_{i})$. But by hypothesis, $\phi_{i}(X_{i})$ is closed in
$X$. Therefore $A\cap\phi_{i}(X_{i})$ is closed in $X$, for each
$i\in\N$.

Let $x\in X-A$ be an arbitrary point and choose an open neighborhood
$V$ of $x$ in $X$ such that
$V\cap\phi_{i}(X_{i})\neq\emptyset\Leftrightarrow i\in \Lambda$,
where $\Lambda\subset\N$ is a finite subset of indices. It follows
that $$V\cap A=\bigcup_{i\in\Lambda}V\cap A\cap\phi_{i}(X_{i}).$$

Now, since each $A\cap\phi_{i}(X_{i})$ is closed in $X$ and $x\notin
A\cap\phi_{i}(X_{i})$, we can choose, for each $i\in\Lambda$, an
open neighborhood $V_{i}\subset V$ of $x$, such that $V_{i}\cap
A\cap \phi_{i}(X_{i})=\emptyset$. Take
$V'=\bigcap_{i\in\Lambda}V_{i}$. Then $V'$ is an open neighborhoodof
$x$ in $X$ and $V'\cap A=\emptyset$. Therefore, $A$ is closed in
$X$.
\end{proof}

\begin{cor}
Let $\{X,\phi_{i}\}$ be a limit space for the CIS $\SIF$ in which
each $X_{i}$ is a compact space. If $X$ is Hausdorff and
$\{\phi_{i}(X_{i})\}_{i}$ is a locally finite cover of $X$, then
$\{X,\phi_{i}\}$ is a fundamental limit space.
\end{cor}
\begin{proof}
Each $\phi_{i}(X_{i})$ is a compact subset of the Hasdorff space
$X$. Therefore, each $\phi_{i}(X_{i})$ is closed in $X$. The result
follows from previous theorem.
\end{proof}

\begin{cor}
Let $\{X,\phi_{i}\}$ be a limit space for the finitely
semicomponible CIS $\SIF$. Then, $\{X,\phi_{i}\}$ is a fundamental
limit space if and only if the collection $\{\phi_{i}(X_{i})\}_{i}$
is a locally finite closed cover of $X$.
\end{cor}
\begin{proof}
Poposition \ref{Prop.Implica.Fechado} and Theorems
\ref{Teorema.Semicomponivel.Fundamental} and
\ref{Teorema.Cobert.Local.Finit.Fechada}.
\end{proof}

Let $f:Z\rightarrow W$ be a continuous map between topological
spaces. We say that $f$ is a {\bf perfect map} if it is closed,
surjective and, for each $w\in W$, the subset $f^{-1}(w)\subset Z$
is compact. (See \cite{Munkres}).

Let $\mathfrak{P}$ be a property of topological spaces. We say that
$\mathfrak{P}$ is a {\bf perfect property} if always that
$\mathfrak{P}$ is true for a space $Z$ and there is a perfect map
$f:Z\rightarrow W$, we have $\mathfrak{P}$ true for $W$. Again, we
say that a property $\mathfrak{P}$ is {\bf countable-perfect} if
$\mathfrak{P}$ is perfect and always that $\mathfrak{P}$ is true for
a countable collection of spaces $\{Z_{n}\}_{n}$, we have
$\mathfrak{P}$ true for the coproduct $\coprod_{n=0}^{\infty}Z_{n}$.
We say that $\mathfrak{P}$ is {\bf finite-perfect} if the previous
sentence is true for finite collections $\{Z_{n}\}_{n=0}^{n_{0}}$ of
topological spaces. It is obvious that every countable-perfect
property is also a finite-perfect property. The reciprocal is not
true. It is also obvious that every perfect property is a
topological invariant.

\begin{Exmp}
The follows one are examples of countable-prefect properties:
Hausdorff axiom, regularity, normality, local compactness, second
axiom of enumerability and Lindel\"of axiom. The compactness is a
finite-perfect property which is not countable-perfect. (For details
see {\rm \cite{Munkres}}).
\end{Exmp}

\begin{thm} \label{XtemP.para.FinitSem}
Let $\{X,\phi_{i}\}$ be a fundamental limit space for the finitely
semicomponible CIS $\SIF$, in which each $X_{i}$ has the
countable-perfect property $\mathfrak{P}$. Then $X$ has
$\mathfrak{P}$.
\end{thm}
\begin{proof}
Let $\{X,\phi_{i}\}$ be a fundamental limit space for $\SIF$. By the
Theorems \ref{EL==LD} and \ref{Unicidade}, there is a unique
homeomorphism $\beta:\widetilde{X}\rightarrow X$ such that
$\phi_{i}=\beta\circ\widetilde{\varphi}$, for all $i\in\N$. Then,
simply to prove that $\widetilde{X}$ has the property
$\mathfrak{P}$, where, remember, $\widetilde{X}=(\coprod
X_{i})/\sim$ is the quotient space constructed in Section
\ref{Section.Inductive.System} (Remember the Remark
\ref{Observacao.X.tilde}).

Consider the quotient map $\rho:\coprod X_{i}\rightarrow
\widetilde{X}$. It is continuous and surjective. Moreover, since the
CIS $\SIF$ is finitely semicomponible, it is obvious that for
$x\in\widetilde{X}$ we have that $\rho^{-1}(x)$ is a finite subset,
and so a compact subset, of $\coprod X_{i}$. Therefore, simply to
prove that $\rho$ is a closed map, since this is enough to conclude
that $\rho$ is a perfect map and, therefore, the truth of the
theorem.

Let $E\subset\coprod X_{i}$ be an arbitrary closed subset of
$\coprod X_{i}$. We need to prove that $\rho(E)$ is closed in
$\widetilde{X}$, that is, that $\rho^{-1}(\rho(E))\cap X_{i}$ is
closed in $X_{i}$ for each $i\in\N$. But note that
$$\rho^{-1}(\rho(E))\cap X_{i}=(E\cap
X_{i})\cup\bigcup_{j=0}^{i-1}f_{j,i-1}(E\cap Y_{j,i-1})
\cup\bigcup_{j=i}^{\infty}f_{i,j}^{-1}(E\cap X_{j+1}),$$  where each
term of the total union is closed. Now, since the given CIS is
finitely semicomponible, there is on the union
$\bigcup_{j=i}^{\infty}f_{i,j}^{-1}(E\cap X_{j+1})$ only a finite
nonempty terms. Thus, $\rho^{-1}(\rho(E))\cap X_{i}$ can be
rewritten as a finite union of closed subsets. Therefore
$\rho^{-1}(\rho(E))\cap X_{i}$ is closed.
\end{proof}

The quotient map $\rho:\coprod X_{i}\rightarrow \widetilde{X}$ is
not closed, in general. To illustrate this fact, we introduce the
follows example:

\begin{Exmp}
Consider the inductive CIS $\{S^{n},S^{n},f_{n}\}$ as in the Example
\ref{Exemplo.S.infinito}, starting at $n=1$. Consider the sequence
of real numbers $(a_{n})_{n}$, where $a_{n}=1/n$, $n\geq1$. Let
$A=\{a_{n}\}_{n\geq2}$ be the set of points of the sequence
$(a_{n})_{n}$ starting at $n=2$. Then, the image of $A$ by the map
$\gamma:[0,1]\rightarrow S^{1}$ given by $\gamma(t)=(\cos t,\sin t)$
is a sequence $(b_{n})_{n\geq2}$ in $S^{1}$ such that the point
$b=(1,0)\in S^{1}$ is not in $\gamma(A)$ and $(b_{n})_{n}$ converge
to $b$. It follows that the subset $B=\gamma(A)$ of $S^{1}$ is not
closed in $S^{1}$. Now, for each $n\geq2$, let $E^{n}$ be the closed
$(n-1)$-dimensional half-sphere imbedded as the meridian into $S^n$
going by point $f_{1,n-1}(b_{n})$. It is easy to see that $E^n$ is
closed in $S^n$ for each $n\geq2$. Let
$E=\bigsqcup_{n=2}^{\infty}E^{n}$ be the disjoint union of the
closed half-spheres $E_{n}$. Then, for each $n\geq2$, $E\cap
S^{n}=E^{n}$, and $E\cap S^1=\emptyset$. Thus, $E$ is a closed
subset of coproduct space $\coprod_{n=1}^{\infty} S^{n}$. However,
$\rho^{-1}(\rho(E))\cap S^{1}=B$ is not closed in $S^{1}$. Hence
$\rho(E)$ is not closed in the sphere $S^{\infty}$. Therefore, the
projection $\rho:\coprod S^{n}\rightarrow(\coprod S^{n})/\sim\,\cong
S^{\infty}$ is not a closed map.
\end{Exmp}

Now, we demonstrate the result of the previous theorem in the case
of stationary CIS's. In this case the result is stronger, and
applies to properties finitely perfect. We started with the
following preliminary result, whose proof is obvious and therefore
will be omitted (left to the reader).

\begin{lem}
Let $\{X,\phi_{i}\}$ be a fundamental limit space for the stationary
CIS $\SIF$. Suppose that this CIS park in the index $n_{0}\in\N$.
Then $\phi_{i}=\phi_{n_{0}}$, for all $i\geq n_{0}$, and
$X\cong\cup_{i=0}^{n_{0}}\phi_{i}(X_{i})$. Moreover, the composition
$$ \xymatrix{\rho_{n_{0}}: \coprod_{i=0}^{n_{0}}X_{i} \ar[r]^{\ \
inc.} & \coprod_{i=0}^{\infty} X_{i} \ar[r]^{ \ \ \ \rho} &
\widetilde{X}}$$ is a continuous surjection, where inc. indicates
the natural inclusion.
\end{lem}

\begin{thm}
Let $\{X,\phi_{i}\}$ be a fundamental limit spaces for the
stationary CIS $\SIF$, in which each $X_{i}$ has the finite-perfect
property $\mathfrak{P}$. Then $X$ has $\mathfrak{P}$.
\end{thm}
\begin{proof}
As in the Theorem \ref{XtemP.para.FinitSem}, simply to prove that
$\widetilde{X}=(\coprod X_{i})/\sim$ has $\mathfrak{P}$.

Suppose that the CIS $\SIF$ parks in the index $n_{0}\in\N$. By the
previous lemma, the map
$\rho_{n_{0}}:\coprod_{i=0}^{n_{0}}X_{i}\rightarrow\widetilde{X}$ is
continuous and surjective. Thus, simply to prove that $\rho_{n_{0}}$
is a perfect map. In order to prove this, it rests only to prove
that $\rho_{n_{0}}$ is a closed map and $\rho_{n_{0}}^{-1}(x)$ is a
compact subset of $\coprod_{i=0}^{n_{0}}X_{i}$, for each
$x\in\widetilde{X}$. This latter fact is trivial, since each subset
$\rho_{n_{0}}^{-1}(x)$ is finite.

In order to prove that $\rho_{n_{0}}$ is a closed map, let $E$ be an
arbitrary closed subset of $\coprod_{i=0}^{n_{0}}X_{i}$. We need to
prove that $\rho^{-1}(\rho_{n_{0}}(E))\cap X_{i}$ is closed in
$X_{i}$ for each $i\in\N$. But note that, as before,
$$\rho^{-1}(\rho_{n_{0}}(E))\cap X_{i}=(E\cap
X_{i})\cup\bigcup_{j=0}^{i-1}f_{j,i-1}(E\cap Y_{j,i-1})
\cup\bigcup_{j=i}^{\infty}f_{i,j}^{-1}(E\cap X_{j+1}),$$  where each
term of this union is closed. Now, since
$E\subset\coprod_{i=0}^{n_{0}}X_{i}$, we have $E\cap
X_{j+1}=\emptyset$ for all $j\geq n_{0}$. Thus, the subsets
$f_{i,j}^{-1}(E\cap X_{j+1})$ which are in the last part of the
union are empty for all $j\geq n_{0}$. Hence,
$\rho^{-1}(\rho_{n_{0}}(E))\cap X_{i}$ is a finite union of closed
subsets. Therefore, $\rho^{-1}(\rho_{n_{0}}(E))\cap X_{i}$ is
closed.
\end{proof}




\end{document}